\documentclass[12pt]{amsart}
\usepackage{BAstyle}
\usepackage[margin=1in]{geometry}
\usepackage{mathrsfs}
\usepackage{enumerate}

\title[Restricting the Splitting Types of a Positive Density Set of Places]{Restricting the Splitting Types of a Positive Density Set of Places in Number Field Extensions}
\author{Brandon Alberts\vspace{-1cm}}

\begin{document}
\sloppy
\maketitle

%\abstract{abstract}

\begin{abstract}
We prove necessary and sufficient conditions for a finite group $G$ with an ordering of $G$-extensions to satisfy the following property: for every positive density set of places $A$ of a number field $K$ and every splitting type given by a conjugacy class $c$ in $G$, $0\%$ of $G$-extensions avoid this splitting type for each $p\in A$.
\end{abstract}

%\tableofcontents

\section{Introduction}\label{sec:introduction}

Let $K$ be a number field, $G\subseteq S_n$ a transitive group, and $\mathcal{F}(K,G)$ the set of degree $n$ extensions $F/K$ with $\Gal(\widetilde{F}/K) \cong G$. There has been significant study into the proportion of $F\in \mathcal{F}(K,G)$ that satisfy certain local conditions. That is, if $\Sigma_p$ is a set of isomorphism classes of $G$-\'etale algebras over $K_p$ for each place $p$ of $K$, what proportion of the elements $F\in \mathcal{F}(K,G)$ satisfy $F\otimes K_p\in \Sigma_p$ for each $p$? More explicitly, determining the value of
\begin{align}\label{eq:proportion}
\lim_{X\to \infty} \frac{\#\{F \in \mathcal{F}(K,G): \forall p,F\otimes K_p \in \Sigma_p,\ {\rm Ht}(F) \le X\}}{\#\{F \in \mathcal{F}(K,G): {\rm Ht}(F) \le X\}},
\end{align}
where ${\rm Ht}:\mathcal{F}(K,G)\to \R^+$ is a height function satisfying the Northcott property: finitely many fields have bounded height. Most often ${\rm Ht}$ is taken to be the discriminant, the product of ramified primes, or another multiplicative counting function as defined by Wood \cite{wood2009}.

We give a partial answer to this question when $\Sigma = (\Sigma_p)$ restricts the splitting type at a positive proportion of places.  Such a $\Sigma$ is called \textbf{nonadmissible} by the author in \cite{alberts2023}, due to the divergence of a certain corresponding Dirichlet series. We measure the nonadmisibility of $\Sigma$ with the following (lower) density:
\[
\delta_{\rm NA}(\Sigma) = \liminf_{x\to \infty}\frac{\#\{|p|\le x : \exists \text{ an unramified }G\text{-\'etale algebra }F_p/K_p\text{ s.t. } F_p\not\in\Sigma_p\}}{\pi_K(x)},
\]
where $|p| := {\rm Nm}_{K/\Q}(p)$ is the norm down to $\Q$ of the prime, and $\pi_K(x) = \#\{p \text{ place of }K: |p|\le x\}$.

\begin{theorem}\label{thm:main}
Let $K$ be a number field, $G\subset S_n$ a transitive group, $\mathcal{F}(K,G)$ the set of degree $n$ extensions $F/K$ with $\Gal(\widetilde{F}/K) \cong G$, and ${\rm Ht}:\mathcal{F}(K,G) \to \R^+$ a height function. Suppose the inverse Galois problem has a positive solution for $G$, that is $\mathcal{F}(K,G) \ne \emptyset$. Then the following are equivalent:
\begin{enumerate}
\item[(i)] For every family of local conditions $\Sigma = (\Sigma_p)$ with $\delta_{\rm NA}(\Sigma) > 0$,
\[
\lim_{X\to \infty} \frac{\#\{F \in \mathcal{F}(K,G): \forall p,F\otimes K_p \in \Sigma_p,\ {\rm Ht}(F) \le X\}}{\#\{F \in \mathcal{F}(K,G): {\rm Ht}(F) \le X\}} = 0.
\]
\item[(ii)] ${\rm Ht}$ admits no accumulating subfields.
\end{enumerate}
\end{theorem}

We say that ${\rm Ht}$ \textbf{admits an accumulating subfield} if there exists a proper normal subgroup $H\normal G$ and a $G/H$-extension $L/K$ such that
\[
\limsup_{X\to\infty}\frac{\#\{F\in \mathcal{F}(K,G) : \widetilde{F}^H = L,\ {\rm Ht}(F) \le X\}}{\#\{F \in \mathcal{F}(K,G) : {\rm Ht}(F) \le X\}} > 0.
\]
In this case, the field $L$ occurs as a subfield of the Galois closure $\widetilde{F}$ for a positive proportion of extensions $F\in \mathcal{F}(K,G)$. Importantly, $L$ is a subfield of the Galois closure but not necessarily of $F$.

\subsection*{Acknowledgements} The author thanks the anonymous referees for helpful suggestions.

\subsection{History}

Previous work on the asymptotic proportion of $G$-extensions satisfying certain local conditions has been primarily restricted to those $G$ for which the asymptotic growth rate of $\#\{F\in \mathcal{F}(K,G) : {\rm Ht}(F) \le X\}$ is known as $X$ tends to infinity. When this growth rate is understood well enough, the techniques can be extended to evaluate the asymptotic rate of growth of the numerator and denominator of (\ref{eq:proportion}) separately.

Such number field counting techniques have been used to compute the limit (\ref{eq:proportion}) in the following cases:
\begin{itemize}
\item $G$ abelian, ${\rm Ht}$ a multiplicative counting function, and $\Sigma$ is such that, for all but finitely many places $p$, $\Sigma_p$ contains all $G$-\'etale algebras over $K_p$ \cite{wood2009}.
\item $G=S_n$ for $n=3,4,5$, ${\rm Ht} = \disc$, and $\Sigma$ is such that, for all but finitely many places $p$, $\Sigma_p$ contains all $G$-\'etale algebras with squarefree discriminant over $K_p$ \cite{bhargava-shankar-wang2015}.
\item $G=S_3$, ${\rm Ht}$ a multiplicative counting function, and $\Sigma$ is such that, for all but finitely many places $p$, $\Sigma_p$ contains all $G$-\'etale algebras over $K_p$ \cite{shankar-thorne2022}.
\item $G = D_4$, ${\rm Ht}$ the conductor, and $\Sigma$ is such that, for all places $p$, $\Sigma_p$ contains all $G$-\'etale algebras of squarefree conductor (with an extra condition at $p=2$) \cite{altug-shankar-varma-wilson2021}.
\end{itemize}
Moreover, under extra hypotheses in the above cases the limit (\ref{eq:proportion}) is proven to be multiplicative of the form
\[
\left(\sum_{(F_p) \in \bigoplus_{p\mid 2} \Sigma_p} \sigma_2((F_p)_{p\mid 2})\right)\prod_{p\nmid 2} \left(\sum_{F_p\in \Sigma_p} \sigma_p(F_p)\right)
\]
for explicit functions $\sigma_p:\Sigma_p \to \R_{\ge 0}$. These extra assumptions require $\{F\in \mathcal{F}(K,G) : \forall p,\ F\otimes K_p\in \Sigma_p\}\ne \emptyset$, ${\rm Ht}$ is a so-called fair counting function if $G$ is abelian, and $G$ is generated by minimal weight elements of ${\rm Ht}$ if $G=S_3$. The distinct behavior at primes above $2$ accounts for certain local conditions that are not realizable by global extensions, called inviable by Wood \cite{wood2009}. When $G$ and ${\rm Ht}$ satisfy these hypotheses, the corresponding cases of Theorem \ref{thm:main} follow by bounding $\Sigma$ above by larger families that, for all but finitely many places, contains all $G$-\'etale algebras of $K_p$.

The primary strength of Theorem \ref{thm:main} is that it applies to all groups $G$ and all heights ${\rm Ht}$, regardless of what is known about the asymptotic growth rate of $\#\{F\in \mathcal{F}(K,G) : {\rm Ht}(F) \le X\}$. It is better compared to theorems that determine number fields by their split primes, for example the following result due to Bauer:

\begin{proposition}[Bauer {\cite[Proposition 13.9]{neukirch2011}}]
If $L/K$ is Galois and $M/K$ is a finite extension, then
\[
\{p\text{ split in }L/K\} \supseteq \{p\text{ lying under at least one degree $1$ place $\beta$ of }M\}
\]
if and only if $L\subseteq M$.
\end{proposition}

Suppose $\Sigma = (\Sigma_p)$ has $\Sigma_p = \{K_p^{\oplus n}\}$ for all $p$ split in some finite extension $M/K$. Then Bauer's result implies
\[
\{F\in \mathcal{F}(K,G) : \forall p,\ F\otimes K_p\in \Sigma_p,\ {\rm Ht}(F)\le X\} \subseteq \{F/K : \widetilde{F}\subseteq M\},
\]
which is necessarily a finite set. With the numerator of (\ref{eq:proportion}) finite, the limit necessarily exists and whether it is zero or not depends on whether there are infinitely many $G$-extensions of $K$ or not. Theorem \ref{thm:main} is best understood as a generalization of results like Bauer's, following from the Chebotarev Density Theorem.

\subsection{Accumulating Subfields}

The presence of accumulating subfields in Theorem \ref{thm:main} is worthy of recognition, as this concept plays an important role in number field counting.

Malle predicted that
\[
\#\{F\in \mathcal{F}(K,G) : |\disc(F/K)|\le X\} \sim c(K,G) X^{1/a(G)}(\log X)^{b(K,G)-1}.
\]
as $X$ tends to infinity \cite{malle2002,malle2004}. Malle gives $a(G)$ and $b(K,G)$ explicitly in terms of the index function $\ind(g) = n = \#\{\text{orbits of }g\}$. This is known to hold when $G$ is abelian or $G=S_n$ with $n=3,4,5$, and is an important step for studying (\ref{eq:proportion}) in these cases. However, it is notable that Malle's conjecture is wrong for some groups and base fields. Kl\"uners famously demonstrated this by proving that $G=C_3\wr C_2\subseteq S_6$ with $K=\Q$ is a counter example \cite{kluners2005}. Malle predicted a growth rate $cX^{1/2}$, however Kl\"uners proved that
\[
\#\{F\in \mathcal{F}(\Q,C_3\wr C_2) : \widetilde{F}^{C_3\times C_3} = \Q(\zeta_3),\ |\disc(F/\Q)|\le X\} \gg X^{1/2}\log X.
\]
The field $\Q(\zeta_3)$ is an accumulating subfield.

There are more subtle examples where accumulating subfields are known to cause issues with heuristic predictions. The leading constant $c(K,G)$ in Malle's prediction is more mysterious than $a(G)$ and $b(K,G)$. In the case that $G=S_n$, Bhargava \cite{bhargava2007} predicted that $c(K,S_n)$ is a particular convergent Euler product. Bhargava's expression can be generalized to other groups $G$, but it is known not to be equal to $c(K,G)$ in some cases with accumulating subfields. Examples include
\begin{itemize}
\item $G=D_4\subseteq S_4$ \cite{cohen-diaz-y-diaz-olivier2002} and $G=H_9\subseteq S_9$ \cite{fouvry-koymans2021} for which $c(\Q,G)$ is explicitely expressed as a convergent sum of special values of $L$-functions, where only the first summand agrees with Bhargava's construction.
\item The cases $G= C_2\wr H$ for certain groups $H$ \cite{kluners2012} and $G = S_n\times A$ for $n=3,4,5$ and $A$ abelian \cite{jwang2021,masri-thorne-tsai-jwang2020}. The authors for these cases do not give an explicit formula for the leading constant, but they do remark that their methods imply that leading constant is given by a sum of Euler products, only one of which agrees with Bhargava's construction.
\end{itemize}
This is not an exhaustive list. At the time of this writing, all known counting results for cases with accumulating subfields are known to either disagree with Malle's predicted growth rate or have a leading coefficient that disagrees with Bhargava's construction.

While accumulating subfields have been known to thwart existing heuristics, previous examples have only been proven in cases which we know more about the growth rate of $\#\{F\in \mathcal{F}(K,G) : |\disc(F/K)|\le X\}$. We will detail the existing heuristics for (\ref{eq:proportion}), which predict the limit will be zero. Theorem \ref{thm:main} does not depend on knowledge of this growth rate, showing that \emph{any} accumulating subfield can violate these heuristics.

\section{Heuristic Predictions}

The Malle--Bhargava principle \cite{bhargava2007,wood2019} generalizes Malle's predictions for counting $G$-extensions by proposing that the generating series
\[
\sum_{\substack{F\in \mathcal{F}(K,G)\\\forall p,\ F\otimes K_p\in \Sigma_p}} |\disc(F/K)|^{-s}
\]
should be arithmetically equivalent to the Euler product
\[
\prod_p \left(\frac{1}{|G|} \sum_{\substack{f_p\in \Hom(G_{K_p},G)\\ F_p\in \Sigma_p}} |\disc(F_p/K_p)|^{-s}\right),
\]
where $F_p/K_p$ is the $G$-\'etale algebra corresponding to $f_p\in \Hom(G_{K_p},G)$. Arithmetically equivalent here means that both series have their rightmost pole at the same place of the same order. Malle's prediction is reproduced via a Tauberian theorem on the Euler product.

The constant terms of each Euler factor is determined by the unramified \'etale algebras in $\Sigma_p$. For all finite places, there are exactly $|G|$ possible unramified homomorphisms $f_p:G_{K_p}/I_p \to G$ corresponding to unramified $G$-\'etale alebras, depending on the image of Frobenius. If $\Sigma_p$ contains all unramified $G$-\'etale algebras then the constant term is $1$, otherwise the constant term is $<1$. In the event that $\Sigma_p$ is nonadmissible, infinitely many $p$ have constant term $<1$. An Euler product of the form
\[
\prod_p \left(c_p + O(p^{-s})\right)
\]
with $c_p\le 1$ for all $p$ and $c_p< 1$ for infinitely many $p$ necessarily diverges to $0$ for all $s$ with positive real part. These cases are not typically considered, in part because it is not clear how one should interpret the divergence. Three possible interpretations come to mind immediately:
\begin{enumerate}
\item Divergence to $0$ predicts that there are zero such extensions.
\item The function $0$ has no poles, which suggests the generating Dirichlet series is holomorphic. A Tauberian theorem would then imply there are at most finitely many such extensions.
\item The Malle--Bhargava principle is only an expression of the ``main term". The divergence to $0$ suggests that $0\%$ of $G$-extensions satisfy these local conditions.
\end{enumerate}
There is some credence to predicting that zero such extensions exist. There are uncountably many ways to choose mutually disjoint nonadmissible $\Sigma$. For each place $p$ for which $\Sigma_p$ does not contain all $G$-\'etale algebras choose between $\Sigma_p$ and the complement $\Sigma_p^{c}$.  Given that there are only countable many number fields, comparing cardinalities implies that uncountably many of these local restrictions are satisfied by exactly zero number fields.

Theorem \ref{thm:main} suggests that the third interpretation applies best to \emph{all} families $\Sigma$, instead of just most families, as long as there are no accumulating subfields. That the Malle--Bhargava principle is only a prediction for the ``main term" can be seen in other settings as well, such as for counting cubic extensions. The number of $S_3$-cubic extensions of $\Q$ with bounded discriminant admits a secondary term of order $X^{5/6}$ \cite{bhargava-shankar-tsimerman2012,taniguchi-thorne2013}, however the corresponding Euler product has no pole at $s=5/6$. This example shows that, at best, the Malle--Bhargava principle can make predictions only for the main term.

\section{The Proof}

We call a sequence of field extensions $F_1,F_2,...,F_m$ containing $K$ independent over $K$ if
\[
\left(\prod_{i\ne j} \widetilde{F}_i\right) \cap \widetilde{F}_j = K
\]
for each $j=1,...,m$. These are precisely the sequences of fields for which
\[
\Gal(\widetilde{F}_1\widetilde{F}_2\cdots \widetilde{F}_m/K) \cong \prod_{i=1}^m \Gal(\widetilde{F}_i/K).
\]

\begin{lemma}\label{lem:independence}
Let $K$ be a number field, $G\subset S_n$ a transitive group, and $\Sigma = (\Sigma_p)$ a family of local conditions with $\delta_{\rm NA}(\Sigma) > 0$. Let $F_1,...,F_m \in \mathcal{F}(K,G)$ be independent over $K$ such that for all places $p$ and each $i=1,2,...,n$, $F_i\otimes K_p\in \Sigma_p$.  Then there exists a constant $C_{|G|,\delta_{\rm NA}(\Sigma)}$ depending only on $|G|$ and $\delta_{\rm NA}(\Sigma)$ such that $m\le C_{|G|,\delta_{\rm NA}(\Sigma)}$.
\end{lemma}

\begin{proof}
For each conjugacy class $c\subseteq G$, set
\[
A_c = \{p : \text{ the unramified \'etale algebra }F_p/K_p\text{ with }\Fr_p(F_p/K_p) \in c\text{ is in }\Sigma_p\}.
\]

We will bound the density of one of these sets using the pigeonhole principle. Certainly the upper density
\[
\delta^+(A_c) := \limsup_{x\to \infty} \frac{\#\{p\in A_c : p\le x\}}{\pi_K(x)} \le 1.
\]
Let $\kappa(G)$ be the number of conjugacy classes in $G$. We can bound
\begin{align*}
\sum_{c} \delta^+(A_c) &= \limsup_{x\to \infty} \frac{1}{\#\pi_K(x)} \sum_{p\le x} |\Sigma_p|\\
&\le \kappa(G) - \liminf_{x\to \infty} \frac{1}{\#\pi_K(x)} \sum_{p\le x} \left(\kappa(G) - |\Sigma_p|\right)\\
&\le \kappa(G) - \liminf_{x\to \infty} \frac{1}{\#\pi_K(x)} \sum_{\substack{p\le x\\\Sigma_p\text{ is not everything}}} 1\\
&= \kappa(G) - \delta_{\rm NA}(\Sigma).
\end{align*}
In particular,
\begin{align*}
\delta_{\rm NA}(\Sigma)\le \sum_{c} (1-\delta^+(A_c))
\end{align*}
is a sum of nonegative numbers. By the pigeonhole principle, at least one of these is larger than $\frac{\delta_{\rm NA}(\Sigma)}{\kappa(G)}$, so that we conclude that there exists a conjugacy class $c$ for which
\[
\delta^+(A_c) \le 1 - \frac{\delta_{\rm NA}(\Sigma)}{\kappa(G)}.
\]

The compositum $M=F_1F_2\cdots F_m$ is necessarily an extension of degree $n^m$ whose Galois closure has Galois group $G^m$ by independence. Moreover, any prime for which $\rho_i(\Fr_p(M/K))\in c$ for at least one projection map $\rho_i:G^m\to G$ necessarily belongs to $A_c$, as each $F_i$ satisfies the local conditions of $\Sigma$. Thus, we compare
\[
\{p : \rho_i(\Fr_p(M/K)) \in c\text{ for at least one }i\} \subseteq A_c.
\]
By the Chebotarev density theorem, we can then compare the (upper) densities
\[
\frac{\#\{(g_i)\in G^m : g_i\in c\text{ for at least one }i\}}{|G|^m} \le 1 - \frac{\delta_{\rm NA}(\Sigma)}{\kappa(G)}.
\]
The lefthand side is given by
\begin{align*}
\frac{\#\{(g_i)\in G^m : g_i\in c\text{ for at least one }i\}}{|G|^m} &= 1 - \frac{\#\{(g_i)\in G^m : g_i\not\in c\text{ for all }i\}}{|G|^m}\\
&= 1 - \frac{(|G| - |c|)^m}{|G|^m}\\
&= 1 - \left(1 - \frac{|c|}{|G|}\right)^m.
\end{align*}
In particular, this implies the lefthand side is an increasing function with
\[
\lim_{m\to \infty}\frac{\#\{(g_i)\in G^m : g_i\in c\text{ for at least one }i\}}{|G|^m} = 1 > 1 - \frac{\delta_{\rm NA}(\Sigma)}{\kappa(G)}
\]
following from $\delta_{\rm NA}(\Sigma) > 0$. Thus, $m$ is bounded in terms of $|c|$, $\kappa(G)$, and $\delta_{\rm NA}(\Sigma)$. Given that there are at most $|G|$ possibilities for $|c|$ and $\kappa(G)$, the dependence on $|c|$ and $\kappa(G)$ can be dropped in exchange for a dependence on $|G|$.
\end{proof}

We now prove Theorem \ref{thm:main}. Suppose first that ${\rm Ht}$ admits no accumulating subfields. Let $F_1,F_2,...,F_m$ be a maximal set of independent $G$-extensions satisfying the local conditions in $\Sigma$, which exists by Lemma \ref{lem:independence}. Then any other $G$-extension $F$ satisfying the local conditions in $\Sigma$ has
\[
\left(\prod_{i=1}^m \widetilde{F}_i\right) \cap \widetilde{F} \ne K.
\]
In particular, the number of possible fields $F$ is bounded by
\[
\sum_{\substack{H\normal G\\H\ne G}}\sum_{K\ne L\le \prod_i \widetilde{F}_i} \#\{F\in \mathcal{F}(K,G) : \widetilde{F}^H = L,\ {\rm Ht}(F) \le X\}.
\]
The sums are finite. As ${\rm Ht}$ admits no accumulating subfields, the result follows.

Conversely, suppose ${\rm Ht}$ has an accumulating subfield given by the proper subgroup $H\normal G$ and a $G/H$ extension $L/K$. Take $\Sigma = (\Sigma_p)$ such that
\[
\Sigma_p = \{F_p/K_p : L\otimes K_p \subseteq \widetilde{F}_p\}.
\]
Any prime which is not split in $L$ necessarily has $K_p^{\oplus n} \not\in \Sigma_p$, so by $H\ne G$ and Chebotarev density we must have $\delta_{\rm NA}(\Sigma) > 0$. However, it is certainly the case that
\[
\{F\in \mathcal{F}(K,G) : \forall p, F\otimes K_p\in \Sigma_p,\ {\rm Ht}(F) \le X\} \supseteq \{F\in \mathcal{F}(K,G) : \widetilde{F}^H = L,\ {\rm Ht}(F) \le X\}.
\]
Given that $L$ is an accumulating subfield, this implies
\begin{align*}
&\limsup_{x\to\infty} \frac{\#\{F\in \mathcal{F}(K,G) : \forall p, F\otimes K_p\in \Sigma_p,\ {\rm Ht}(F) \le X\}}{\#\{F\in \mathcal{F}(K,G) : {\rm Ht}(F) \le X\}}\\
&\ge \limsup_{x\to\infty} \frac{\#\{F\in \mathcal{F}(K,G) : \widetilde{F}^H = L,\ {\rm Ht}(F) \le X\}}{\#\{F\in \mathcal{F}(K,G) : {\rm Ht}(F) \le X\}}\\
&> 0.
\end{align*}

\bibliographystyle{alpha}
%\bibliography{../../BAreferencesV2020.bib}
\bibliography{Restricting_Positive_Density_of_Splitting_Types.bbl}
\end{document}